\documentclass[preprint,11pt]{elsarticle}

\usepackage{natbib}
\usepackage{graphicx,amsmath,amsthm,amsfonts,color,amssymb,comment,caption,subcaption}
\usepackage{tikz}
\usepackage{subfiles} 
\usepackage[nameinlink]{cleveref} 
\usepackage{kpfonts}
\usepackage{etoolbox}
\usepackage{thmtools}
\usepackage{environ}
\usepackage[textsize=tiny]{todonotes}

\def\bfm#1{\boldsymbol{#1}} 
\def\RR{\mathbb{R}}

\newcommand{\overbar}[1]{\mkern 1.5mu\overline{\mkern-1.5mu#1\mkern-1.5mu}\mkern 1.5mu}

\newtheorem{thm}{Theorem}
\newtheorem{lem}[thm]{Lemma}

\newtheorem{rem}[thm]{Remark}
\newproof{pf}{Proof}
\newproof{pot}{Proof of Theorem \ref{thm2}}
\journal{Computer Aided Geometric Design}

\begin{document}

\begin{frontmatter}

\title{Optimal one-sided approximants of circular arc}

\author[address1,address2]{Ada \v{S}adl Praprotnik\corref{corauth}}
\ead{ada.sadl-praprotnik@imfm.si}
\cortext[corauth]{Corresponding author}

\author[address1,address2]{Ale\v{s} Vavpeti\v{c}}
\ead{ales.vavpetic@fmf.uni-lj.si}

\author[address1,address2]{Emil \v{Z}agar}
\ead{emil.zagar@fmf.uni-lj.si}

\address[address1]{Faculty of Mathematics and Physics, University of Ljubljana, 
Jadranska 19, Ljubljana, Slovenia}
\address[address2]{Institute of Mathematics, Physics and Mechanics, Jadranska 19, Ljubljana, Slovenia}

\begin{abstract}
 The optimal one-sided parametric polynomial approximants of a 
 circular arc are considered. More precisely, the approximant 
 must be entirely in or out of the underlying circle of an arc.
 The natural restriction to an arc's approximants interpolating
 boundary points is assumed. However, the study
 of approximants, which additionally interpolate corresponding tangent
 directions and curvatures at the boundary of an arc, is also
 considered.
 Several low-degree polynomial approximants are studied in detail.
 When several solutions fulfilling the
 interpolation conditions exist, the optimal one is
 characterized, and a numerical algorithm for its construction is
 suggested. Theoretical results are demonstrated with several numerical
 examples and a comparison with general (i.e. non-one-sided)
 approximants are provided.
\end{abstract}

\begin{keyword}
geometric interpolation  \sep circular arc \sep parametric polynomial \sep B\'ezier curve
\MSC[2020] 65D05 \sep  65D07 \sep 65D17
\end{keyword}

\end{frontmatter}


\section{Introduction}
Circular arcs are among the most fundamental 
geometric objects used in the design of planar curves. Joining two arcs
together in a biarc is quite a popular procedure 
in engineering, specially in robotics (see \cite{Sir-Feichtinger-Juettler-biarcs-2006} and references therein). 
Although circular arcs can be considered as canonical objects in a particular computer-based 
design system, it is often desirable to work with polynomial objects
only. Since a circular arc does not possess an exact polynomial
representation, this leads to the problem of good 
(or the optimal) approximation 
of circular arcs by parametric polynomial curves. Several authors 
studied this problem in the frame of computer-aided geometric
design (CAGD), from the pioneering paper 
\cite{Dokken-Daehlen-Lyche-Morken-90-CAGD} to the recent 
results on optimal approximation in  
\cite{Vavpetic-2020-optimal-circular-arcs} 
and \cite{Vavpetic-Zagar-21-circle-arcs-Hausdorff}.
Meanwhile, many authors considered a relaxed problem. 
They have either
used the simplified radial distance as a measure of optimality,
or they have fixed some free parameters to make the problem easier to handle
(see \cite{Vavpetic-Zagar-2019-framework-circle-arcs} and a
comprehensive list of references therein).\\
In this paper, we focus on the optimal one-sided approximation, i.e., 
the one where an approximant of a circular arc is entirely in or
entirely out of the underlying circle. 
This might be of interest when
an approximant must not cross a circular arc for some
reason (e.g., a design of circular molds). In both cases, the inner
and outer ones will be studied in detail, and
optimal solutions will be provided according to the Hausdorff distance.\\
The paper is organized as follows. In Section \ref{sec:preliminaries}, we review some basic definitions and discuss the number of free parameters for 
approximants of different degrees and the order of the geometric continuity. Sections \ref{sec:G20}, \ref{sec:G31}, and \ref{sec:G42} deal with one--parametric optimization problems related the optimal one-sided $G^0, G^1$ and $G^2$ approximants of degree $2, 3$ and $4$, respectively. In Sections \ref{sec:G30} and \ref{sec:G41}, we move further to the two parametric problems and construct the optimal one-sided cubic $G^0$ and quartic $G^1$ approximants. Some numerical 
examples and comparison of the obtained results are presented in Section  \ref{sec:examples}. The paper concludes with Section \ref{sec:conclusion} where we give some final remarks and suggestions for possible future work.

\section{Preliminaries}
\label{sec:preliminaries}

Let $\bfm{c}:[-\varphi, \varphi] \rightarrow \RR^2$, $0 <\varphi \leq \frac{\pi}{2}$, denote the standard parameterization of a canonical
unit circular arc given by 
$\bfm{c}(\alpha) = (\cos \alpha, \sin \alpha)^T$.  
Note that any other
arc can be easily translated, rotated and scaled to the canonical one, and these transformations do not affect the optimality.
We are looking for a polynomial 
approximant $\bfm{p}$ of $\bfm{c}$, where $\bfm{p}:[-1,1] \rightarrow \RR^2$ is a parametric polynomial curve of degree $n$ that interpolates two boundary points of $\bfm{c}$, i.e., $\bfm{p} (\mp 1) = (\cos \varphi, \mp \sin\varphi)^T$. Let
\begin{equation*}
    \psi(t) = \left\|\mathbf{p}(t)\right\|_2^2-1, 
    \quad t \in [-1,1],
\end{equation*}
where $\left\|.\right\|_2$ denotes the Euclidean norm, be the simplified radial
error function. Our goal is to find $\bfm{p}$
which minimizes $\left|\psi\right|$ over $[-1,1]$. In addition,
we shall require that $\psi$ is either non-positive or non-negative,
and refer to the corresponding $\bfm{p}$ as an inner or an outer approximant, respectively. Note that the radial error
is given by $\psi_r=\sqrt{\psi+1}-1$, and it implies the Hausdorff distance 
(see \cite{Jaklic-Kozak-2018-best-circle}, e.g.). For the
inner and outer approximation, we also observe that 
optimality
of $\psi_r$ coincides
with optimality of $\psi$, thus it is
enough to consider $\psi$ only.\\
It will turn out convenient to express the polynomial curve 
$\bfm{p}$ in the Bernstein-Bézier form as
\begin{equation*}
    \bfm{p}(t) = \sum_{i=0}^n B_i^n(t)\bfm{b}_i, \quad t \in [-1,1],
\end{equation*}
where $\bfm{b}_i$, $i=0,1,\ldots,n$, are control points and $B_i^n$, $i=0,1,\ldots,n$, are reparameterized Bernstein basis polynomials over interval $[-1,1]$, i.e.,
\begin{equation*}
    B_i^n(t) = \binom{n}{i}\left(\frac{1+t}{2}\right)^i\left(\frac{1-t}{2}\right)^{n-i}.
\end{equation*}
The interpolation of boundary points implies that the interpolant
$\bfm{p}$ is a $G^0$ approximant. If we additionally require that
also tangent directions are interpolated, then we deal with $G^1$
approximant, and if we further interpolate the signed curvature,
then the approximant is called a $G^2$ approximant.\\
Since a circular arc $\bfm{c}$ is symmetric across the abscissa,
the approximant $\bfm{p}$ must possess this property too. 
This reduces the number of free parameters
in $\bfm{p}$ from $2n+2$ to $n+1$. Since $\bfm{p}$ is a $G^0$
approximant, the number of free parameters is further reduced to $n-1$.
Since $\psi$ depends on these $n-1$ free parameters, our goal is
to determine them in a way that the corresponding $\bfm{p}$
minimizes
\begin{equation}\label{eq:m_p}
    \max_{t\in[-1,1]}\left|\psi(t)\right|.
\end{equation} 
The number of free parameters depends on the degree $n$
of the approximant $\bfm{p}$, and on the order of geometric 
continuity. The bigger the difference between the degree and
the order is, the harder the optimization of \eqref{eq:m_p}. 
Our discussion will first deal with the cases where the difference
is two, followed by the analysis of cases where it is three.
However, rising the degree $n$ also makes the problem much more 
difficult to handle, so we shall restrict it to degrees up to four.\\
In the following analyses, $\cos \varphi$ and $\sin \varphi$
will often be involved, and we will use the abbreviations
$c=\cos \varphi$ and $s=\sin\varphi$. Note that $0\leq c<1$
and $0<s\leq 1$.

\section{Optimal quadratic $G^0$ approximants}
\label{sec:G20}

In the quadratic $G^0$ case control points determining the approximant $\bfm{p}$ are given as
$\bfm{b}_0 = (c,-s)^T$, 
$\bfm{b}_1 = (\xi,0)^T$, and $\bfm{b}_2 = (c,s)^T$,
where $\xi > 0$. The function $\psi$ is a quartic polynomial
depending only on the parameter $\xi$ and can be expressed as
\begin{align*}
    \psi(t,\xi) &= 
   \frac{1}{8} \left(1-t^2\right) \left(2 \left(1-t^2\right) \xi ^2+4 c \left(1+t^2\right)\xi -8+2 c^2 \left(1-t^2\right)\right).
\end{align*}
We observe that $\psi(\cdot,\xi)$ increases on $(0,\infty)$
for all $t\in (-1,1)$ and $0\leq c <1$.
Thus, the optimal inner approximant is uniquely determined by the maximal $\xi>0$ 
for which $\max_{t\in(-1,1)}\psi(t,\xi)=0$ and the
optimal outer approximant is also uniquely given by the minimal $\xi>0$
for which $\min_{t\in(-1,1)}\psi(t,\xi)=0$. Let us denote them as $\xi^\lor$ and $\xi^\land$,
respectively.\\
Since $\psi(\cdot,\xi)$ is symmetric on $(-1,1)$, the parameter $\xi^\lor$ is obtained 
as a solution of the equation $\psi(0,\xi^\lor)=0$, i.e., $\xi^\lor=2-c$ and 
the parameter $\xi^\land$ is determined as a solution of the
equation $\psi'(1,\xi)=0$, i.e., $\xi^\land=\tfrac{1}{c}$. Note that in the latter case 
we get the $G^1$ approximant.

\section{Optimal cubic $G^1$ approximants}
\label{sec:G31}

In this case, the control points of an approximant $\bfm{p}$ are given as 
\begin{equation*}
    \bfm{b}_0 = (c,-s)^T, \quad \bfm{b}_1 = (c,-s)^T+\xi(s,c)^T, \quad \bfm{b}_2 = (c,s)^T+\xi(s,-c)^T \quad \text{and} \quad \bfm{b}_3 = (c,s)^T,
\end{equation*}
where $\xi > 0$. Note that the $G^1$ condition at the boundary points has already
been incorporated. The error function $\psi$ is a polynomial of degree $6$, which can
be written as
\begin{align*}
    \psi(t,\xi) 
   &=\frac{1}{16} \left(1-t^2\right)^2 \left(9 \xi^2 \left(c^2t^2 +1-c^2\right)+12 \xi
   \left(2-t^2\right) c \sqrt{1-c^2}-4 \left(4-t^2\right) \left(1-c^2\right)\right).
\end{align*}
Again, $\psi(t,\cdot)$ is an increasing function for all $t\in (-1,1)$ and $0\leq c<1$. Since in addition $\psi(\cdot,\xi)$ is 
a polynomial of degree six with a positive leading coefficient, 
the parameter $\xi^\land$ of the optimal outer approximant must fulfil the condition
$\psi(0,\xi^\land)=0$, and consequently $\xi^\land=\tfrac{4\sqrt{1-c^2}}{3(1+c)}$.\\ 
On the other hand, the condition for the approximant being an inner one is
that $t=\pm 1$ are triple zeros of $\psi$, i.e.,  $\psi''(\pm 1,\xi)=0$ which implies the solution 
$\xi^\lor=\tfrac{2}{3}\sqrt{1-c^2}(\sqrt{3+c^2}-c)$.

\section{Optimal quartic $G^2$ approximants}
\label{sec:G42}
As we will see in the following, the analysis gets
more complicated here as in the previous two cases.
Considering $G^1$ and $G^2$ conditions, we can
express control points of an approximant $\bfm{p}$ as
\begin{equation*}
    \bfm{b}_0 = (c,-s)^T, \bfm{b}_1 = (c,-s)^T+\xi (s,c)^T, 
    \bfm{b}_2=\left(\tfrac{3-4\xi^2}{3c},0\right)^T,
    \bfm{b}_3 = (c,s)^T+\xi(s,-c)^T, \bfm{b}_4 = (c,s)^T,
\end{equation*}
where again $\xi > 0$. The free parameter $\xi$ was chosen similarly as in the 
previous section. The control point $\bfm{b}_2$ is not defined for $c=0$,
but this case can be obtained as the limit $c\to 0$.
The function $\psi$ is now a polynomial of degree $8$, and
with 
    $\beta_{1,2}=\tfrac 1 2\left(c\sqrt{1-c^2}-\sqrt{3\pm 8c+6c^2-c^4}\right)$, 
    $\beta_{3,4}=\tfrac 1 2\left(c\sqrt{1-c^2}+\sqrt{3\mp 8c+6c^2-c^4}\right)$,
    $ \alpha_{1,2}=\mp\tfrac 1 2 \sqrt{1-c^2}\left(\sqrt{3+c^2}\pm c\right)$,
    it can be written  as
\begin{equation*}
    \psi(t,\xi)=\frac{(t^2-1)^3}{64c^2}
    \left(16(\xi-\alpha_1)^2(\xi-\alpha_2)^2t^2
    -16(\xi-\beta_1)(\xi-\beta_2)(\xi-\beta_3)(\xi-\beta_4)\right),
\end{equation*}
so its leading coefficient is $\text{lc}(\psi(\cdot,\xi))=\tfrac{1}{64c^2}16(\xi-\alpha_1)^2(\xi-\alpha_2)^2$.
Unfortunately, $\psi(t,\cdot)$ is no longer a monotone function, and a more involved analysis is needed.\\
Observe first that 
$\psi(0,\xi)=\tfrac{1}{4c^2}(\xi-\beta_1)(\xi-\beta_2)(\xi-\beta_3)(\xi-\beta_4)$
and
$\frac{\partial{\psi}}{\partial\xi}(0,\xi)=\tfrac{1}{c^2}(\xi-\gamma_1)(\xi-\gamma_2)(\xi-\gamma_3)$,
where $\gamma_2=\tfrac1 2c\sqrt{1-c^2}$ and
$\gamma_{1,3}=\tfrac1 2(c\sqrt{1-c^2}\mp\sqrt{3+6c^2-c^4})$.
Let 
$$
\text{td}(\xi):=\psi'''(1,\xi)
=\frac{48 \sqrt{1-c^2} \xi ^3}{c}
  +48 \xi ^2-12 \sqrt{1-c^2} \left(\frac{3}{c}+c\right) \xi +12 \left(1-c^2\right).
$$
Then $\text{td}(0)=12(1-c^2)>0$, $\text{td}'(0)=-\frac{12}{c}\sqrt{1-c^2}(3+c^2)<0$,
and $\text{td}(\beta_3)=12 (1-c)^4 \left((2+c)-\sqrt{1+c} \sqrt{3+c}\right)^2>0$.
Furthermore, $\text{td}'(\xi)=\tfrac{144}{c}\sqrt{1-c^2}(\xi-\delta_1)(\xi-\delta_2)$,
where
$\delta_{1,2}=\tfrac{1}{6\sqrt{1-c^2}}(-2c\mp \sqrt{9-2c^2-3c^4})$, and 
$$ 
  \text{td}(\delta_2)=
  -\frac{48 \left(1-c^2\right)^3 \left(27-18 c^2-c^4\right)}
  {c \left(8 c \left(27 \left(1-c^2\right)^2+22 c^2-18 c^4\right)+4 
  \left(9-2 c^2-3 c^4\right) \sqrt{9-2c^2-3 c^4}\right)}<0.
$$
It is easy to see that $\alpha_1,\beta_1,\gamma_1,\delta_1\le 0$
and $\beta_2\le\gamma_2\le\delta_2\le\beta_3\le \alpha_2\le\gamma_3\le\beta_4$ 
for all $c$ and $\beta_2\ge 0$ if and only if $c\ge\tfrac 3 5$.

\subsection{Optimal inner approximant}\label{subsec:G24_inner}

For $\xi^\lor=\beta_3$ we have $\psi(t,\xi^\lor)=-\tfrac 1 4(1-c)^4(2+c-\sqrt{3+4c+c^2})^2t^2(1-t^2)^3$. A parameter $\xi>\beta_4$ does not induce an inner approximant since $\psi(0,\xi)>0$. Suppose $\beta_3<\xi\le \beta_4$. Since $\text{td}$ is increasing on $[\beta_3,\beta_4]$ there is $\varepsilon>0$, such that $\psi(t,\xi)<\psi(t,\xi^\lor)$ for $t\in (1-\varepsilon, 1)$. If $\xi$ induces a better approximant as $\xi^\lor$, then the graphs of $\psi(\cdot,\xi)$ and $\psi(\cdot,\xi^\lor)$ have two intersection on $(0,1)$, therefore ten intersections on $[-1,1]$, which is not possible. If $c<\tfrac 3 5$, there are no $\xi\in[0,\xi^\lor]$ which induce an inner approximant.
Suppose $c\ge\tfrac 3 5$ and $0\le \xi\le \beta_2$. Since 
$\text{lc}(\psi(\cdot,\xi^\lor))<\text{lc}(\psi(\cdot,\beta_2))$, 
the parameter $\xi^\lor$ induces a better approximant then $\beta_2$. Since $\text{td}$ is decreasing on $[0,\beta_2]$ a parameter $\xi<\beta_2$ induces a worse approximant then $\beta_2$. Therefore, $\xi^\lor$ induces the optimal inner approximant.
In the limit case, when $c=0$, we get $\xi^\lor=\tfrac{\sqrt{3}}{2}$ and the abscissa of the control point $\bfm{b}_2$ is 
$\lim_{c\to 0}\tfrac{3-4\beta_3^2}{3c}=\tfrac 2 3(4-\sqrt{3})$.

\subsection{Optimal outer approximant}

Let $\xi^\land$ be the only zero of $\text{td}$ on the interval $[\delta_2,\infty)$. By the above calculations $\xi^\land<\beta_3$, so $\xi^\land$ induces an outer approximant. If $\xi>\xi^\land$, then $\text{td}(\xi)>0$, hence $\xi$ does not induce an outer approximant. Suppose that $0\le \xi<\xi^\land$. If $\text{td}(\xi)<0$, then $\xi$ does not induce an outer approximant. Let $\text{td}(\xi)\ge 0$. Since 
$\text{lc}(\psi(\cdot,\xi))>\text{lc}(\psi(\cdot,\xi^\land))$ the graphs of $\psi(\cdot,\xi)$ and $\psi(\cdot,\xi^\land)$ intersect on $(1,\infty)$. Hence, the graphs can not have an intersection on the interval $(-1,1)$, so $\xi$ induces a worse approximant as $\xi^\land$.
In the limit case, when $c=0$, we get $\xi^\land=\tfrac{\sqrt{3}}{2}$ and 
the abscissa of the control point $\bfm{b}_2$ is 
$\tfrac {8\sqrt{3}}{9}$.\\
In all previous cases, only one parameter was involved in the
optimization process. Now we move to the two--parametric
optimization.

\section{Optimal cubic $G^0$ approximants}
\label{sec:G30}
This section deals with cubic one-sided approximants of $\bfm{c}$. The control points determining the approximant $\bfm{p}$ are
\begin{equation*}
    \bfm{b}_0 = (c,-s), \quad \bfm{b}_1 = (\xi,-\eta), \quad \bfm{b}_2 = (\xi,\eta) \quad \text{and} \quad \bfm{b}_3 = (c,s),
\end{equation*}
where $\xi>0$. The function $\psi$ is of degree $6$ and it is of the form
\begin{equation*}
    \psi(t,\xi,\eta) =\frac{1}{16} \left(t^2-1\right) \left(\left(3 \eta -\sqrt{1-c^2}\right)^2 t^4+\left(16 \left(1-c^2\right)-9   \left(\eta +\sqrt{1-c^2}\right)^2+9 (\xi -c)^2\right) t^2+\left(16- \left(3\xi+c\right)^2\right)\right),
\end{equation*} 
hence
\begin{equation}
    \label{eq:cubicpsi_lc_0_der1}
    \text{lc}(\psi(\cdot,\xi,\eta))= \frac{1}{16}\left(\sqrt{1-c^2}-3\eta\right)^2,
    \quad
    \psi(0,\xi,\eta) = \frac{1}{16}\left(3\xi+c \right)^2-1, 
    \quad
    \psi'(1,\xi,\eta) = 3\left(1-\xi c-\eta\sqrt{1-c^2}\right).
\end{equation}
We deal with the two-parametric optimization problem here, and the analysis
is much more involved than in the previous cases. Let us consider the inner and the outer
approximation separately again.

\subsection{Optimal inner approximant}
Following the idea of equioscillation of
the best polynomial approximant in the functional case, we might expect 
equioscillation of $\psi$. Since
it is an inner approximant, it should touch the abscissa from below
as many times as possible. Thus, the guess is that its graph
looks like the one in \Cref{fig:graph_psi_cubic_quartic} (left).
Since $\psi$ is of degree $6$ and it has two simple zeros at
$\pm 1$, the remaining two zeros on $(-1,1)$
should be double ones and
related to the maxima of the Chebyshev polynomial of degree six. 
Consequently, they equal to $\pm \tfrac1 2$ and we must have
\begin{equation}
    \label{eq:cubic_G0_system}
    \psi\left(\tfrac{1}{2},\xi,\eta\right) = 0, \quad  \psi'\left(\tfrac{1}{2},\xi,\eta\right) = 0.
\end{equation}
It will be shown that the above system has a unique solution with $\xi>c$ and
that this solution implies the optimal approximant. Observe that \eqref{eq:cubic_G0_system} is an algebraic system of equations for $\xi$ and $\eta$. Two polynomials of a Gr\"obner basis for the ordering $\eta\prec \xi$
are $p_1(\cdot)=(\cdot-c)g_1(\cdot)$ and $g_2$, where
\begin{align*}
  g_1(\xi)&=2187\xi^3+3807c\xi^2+(1953c^2-3840)\xi+245c^3-4352c,\\
  g_2(\xi,\eta)&=-8 \left(1-c^2\right) \eta +\sqrt{1-c^2} \left(27 \xi ^2+10 c \xi +3 c^2-40\right).
\end{align*}
The zero $\xi=c$ of $p_1$ is not admissible, since the corresponding
polynomial approximant reduces to the line segment, which is not the
optimal one. Thus, we are restricted to the positive zeros of $g_1$. Let 
$I=(\xi_{min},\xi_{max})$, where
    \begin{equation*}
      \xi_{min}=\frac{1}{27} \left(-5 c+8 \sqrt{2} \sqrt{9-c^2}\right)
      \quad {\rm and} \quad
      \xi_{max}=\frac{4-c}{3}.
    \end{equation*}
Note that $c<1<\xi_{min}<\xi_{max}$ for all $c\in [0,1)$.
\begin{lem}\label{eq:g1_cubic}
    The only positive zero of $g_1$ is in $I$.
\end{lem}
\begin{proof}
    We have
    $g_1\left(-\tfrac{37}{27}c\right)=\tfrac{8192}9c(1-c^2)\ge 0$,
    $g_1(0)=c(245c^2-4352)\le 0$,
    $g_1(\xi_{min})=-\frac{2048(1-c^2)^2}{4 c+\sqrt{2} \sqrt{9-c^2}}< 0$,
    and $g_1(\xi_{max})=64(1-c)^3>0$.
    Since the leading coefficient of $g_1$ is positive,
    $g_1$ has two non-positive zeros and the unique positive one
    in $I$. This completes the proof.
\end{proof}
Since $g_2(\cdot,\eta)$ is a linear polynomial, the unique solution of 
the system \eqref{eq:cubic_G0_system} is $(\xi^\lor,\eta^\lor)$, 
where $\xi^\lor$ is the unique zero of $g_1$ in $I$, and 
\begin{equation}\label{eq:eta_opt_G30}
  \eta^\lor=\frac{27 {\xi^\lor}^2+10 c \xi^\lor +3 c^2-40}{8\sqrt{1-c^2}}.
\end{equation}
Let $J=(\eta_{min},\eta_{max}$), where
$$
  \eta_{min}=\frac{\sqrt{1-c^2}}{3} 
  \quad
  {\rm and}
  \quad
  \eta_{max}=\frac{3-c}{1+c}\frac{\sqrt{1-c^2}}{3}.
$$
\begin{lem}\label{lem:etaG30}
  If $(\xi^\lor,\eta^\lor)$ is the solution of the system \eqref{eq:cubic_G0_system},
  then $\eta^\lor \in J$.
\end{lem}
\begin{proof}
  Let $f(\xi)=\frac{27 {\xi}^2+10 c \xi +3 c^2-40}{8\sqrt{1-c^2}}$.
  Then
  $f'(\xi)=\frac{27 \xi +5 c}{4 \sqrt{1-c^2}}$ 
  and $f$ is strictly increasing on $I$. Thus 
  $f(\xi_{min})=
  \frac{\sqrt{1-c^2}}{3}<f(\xi) < \frac{3-c}{1+c}\frac{\sqrt{1-c^2}}{3}=
  f(\xi_{max})$.
  Since $\xi^\lor\in I$ and
  $\eta^\lor=f(\xi^\lor)$ by \eqref{eq:eta_opt_G30},
  the proof is complete.
\end{proof}
\noindent We shall see in the following that $(\xi^\lor,\eta^\lor)$
induces the polynomial approximant $\bfm{p}$ with 
\begin{equation*}
  (\xi^\lor,\eta^\lor)=
  \underset{\substack{(\xi,\eta)}}{\rm argmin}
  \left(\max_{t\in[-1,1]}\left|\psi(t,\xi,\eta)\right|\right),
\end{equation*}
i.e., the optimal one.\\
Let us suppose that $(\overbar{\xi},\overbar{\eta})$ induces a better
approximant. If $\overbar{\xi}<\xi^\lor$, then by 
\eqref{eq:cubicpsi_lc_0_der1} we have
$\psi(0,\overbar{\xi},\overbar{\eta})\leq\psi(0,\xi^\lor,\eta^\lor)<0$.
Thus $\left|\psi(0,\overbar{\xi},\overbar{\eta})\right|
\geq\left|\psi(0,\xi^\lor,\eta^\lor)\right|$, and 
$\psi(\cdot,\overbar{\xi},\overbar{\eta})$ does not provide a smaller error than 
$\psi(\cdot,\xi^\lor,\eta^\lor)$.

If $\overbar{\xi}\geq\xi^\lor$ and $\overbar{\eta}>\eta^\lor$, then by \eqref{eq:cubicpsi_lc_0_der1} 
$$
  \psi'(1,\overbar{\xi},\overbar{\eta})<\psi'(1,\xi^\lor,\eta^\lor)\quad {\rm and} \quad
   \text{lc}(\psi(\cdot,\overbar{\xi},\overbar{\eta}))
   >
   \text{lc}(\psi(\cdot,\xi^\lor,\eta^\lor)).
$$
Consequently, $\psi(\cdot,\overbar{\xi},\overbar{\eta})$ and
$\psi(\cdot,\xi^\lor,\eta^\lor)$ must intersect six times in $[-1,1]$
(see \Cref{fig:graph_psi_intersections} (left)) and at least twice out of $[-1,1]$, which contradicts the fact that $\psi(\cdot,\overbar{\xi},\overbar{\eta})$ is of degree six. If $\overbar{\xi} = \xi^\lor$ and $\overbar{\eta}<\eta^\lor$, then $\psi'(1,\overbar{\xi},\overbar{\eta})>\psi'(1,\xi^\lor,\eta^\lor)$, so $\psi(\cdot,\overbar{\xi},\overbar{\eta})$ and $\psi(\cdot,\xi^\lor,\eta^\lor)$
 must intersect at least eight times in $[-1,1]$ 
 (see \Cref{fig:graph_psi_intersections} (left) again), which can not be the case.\\
 Thus, we are left with
 $\overbar{\xi}>\xi^\lor$ and $\overbar{\eta}\leq \eta^\lor$. 
 The case where $\psi'(1,\overbar{\xi},\overbar{\eta})\geq \psi'(1,\xi^\lor,\eta^\lor)$ is eliminated by using the same argument as in the previous case, and the following lemma eliminates the last case, i.e., $\psi'(1,\overbar{\xi},\overbar{\eta})< \psi'(1,\xi^\lor,\eta^\lor)$.
 
\begin{figure}[htbp]
    \centering
    \begin{subfigure}{.5\textwidth} 
        \centering
        \includegraphics[scale=0.8]{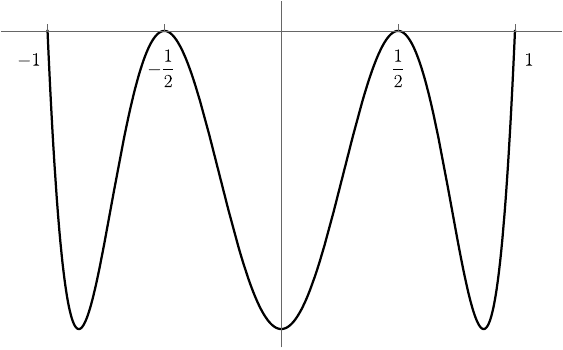}
    \end{subfigure}%
    \begin{subfigure}{.5\textwidth} 
         \centering
         \includegraphics[scale=0.8]{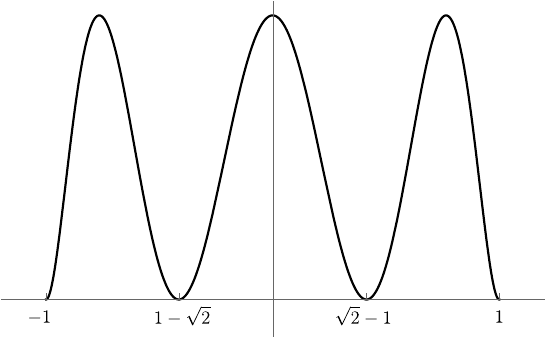}
    \end{subfigure}%
    \caption{Graph of error function $\psi$ in cubic $G^0$ case (left) and quartic $G^1$ case (right).}
    \label{fig:graph_psi_cubic_quartic}
\end{figure}

\begin{figure}[htbp]
    \centering
    \begin{subfigure}{.5\textwidth} 
        \centering
        \includegraphics[scale=0.8]{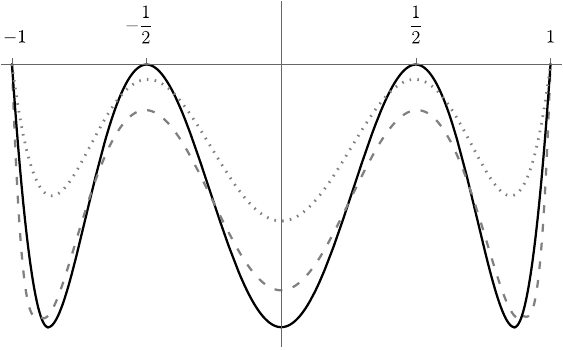}
    \end{subfigure}%
    \begin{subfigure}{.5\textwidth}
         \centering
         \includegraphics[scale=0.8]{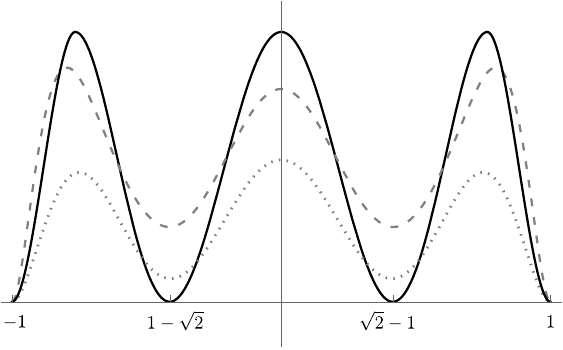}
    \end{subfigure}%
    \caption{Potential optimal approximants
    in the cubic $G^0$ case (left) and in the quartic $G^1$ case
    (right).}
    \label{fig:graph_psi_intersections}
\end{figure}

\begin{lem}
\label{lem:cubicG0}
    If $\overbar{\xi}>\xi^\lor$, 
    $\overbar{\eta}\leq \eta^\lor$ and 
    $\psi'(1,\overbar{\xi},\overbar{\eta})< \psi'(1,\xi^\lor,\eta^\lor)$, then 
    $\psi\left(\tfrac{1}{2},\overbar{\xi},\overbar{\eta}\right)>0$ and
    $\psi\left(\cdot,\overbar{\xi},\overbar{\eta}\right)$ is not an inner approximant.
\end{lem}

\begin{proof}
    Since the pair of parameters $(\overbar{\xi},\overbar{\eta})$ provides
    an inner approximant, we have $\overbar{\xi}\leq \xi_{max}$. The
    assumptions of the lemma then imply
    \begin{equation}\label{eq:G30_triangle}
       \overbar{\xi}\in\left(\xi^\lor,\xi_{\max}\right],\quad 
       \overbar{\eta}\in \left(
       -\frac{c}{\sqrt{1-c^2}}\overbar{\xi}
       +\frac{c}{\sqrt{1-c^2}}\xi^\lor+\eta^\lor,\eta^\lor\right].
    \end{equation}
   Thus $(\overbar{\xi},\overbar{\eta})\in {\mathcal T}$,
   where ${\mathcal T}$ is the triangle given by \eqref{eq:G30_triangle}
   (see \Cref{fig:ellipse} (left)).
   We need to show that $\psi(\tfrac{1}{2},\overbar{\xi},\overbar{\eta}) > 0$
   for $(\overbar{\xi},\overbar{\eta}) \in \mathcal{T}$,
   or, equivalently, that $\mathcal{T}$ lies outside of the ellipse $\mathcal{E}$
   given by $\psi(\tfrac{1}{2},\overbar{\xi},\overbar{\eta})=0$. 
   The implicit equation of the ellipse reads 
    \begin{equation}
        \label{eq:G30_ellipse}
        4(9\overbar{\xi}+7c)^2+(9\overbar{\eta}+13\sqrt{1-c^2})^2-1024= 0,
    \end{equation}
    revealing that its centre is in the third quadrant. Thus, it is concave and decreasing in the first quadrant. The triangle ${\mathcal T}$ is above 
    the line $\ell_{\mathcal T}$ given by the equation
    \begin{equation}
        \label{eq:line_ell_T}
        \eta = -\frac{c}{\sqrt{1-c^2}}\xi 
        + \frac{c}{\sqrt{1-c^2}}\xi^\lor + \eta^\lor.
    \end{equation}
    Let $\ell_{\mathcal E}$ be the tangent line of the ellipse ${\mathcal E}$ at
    $(\xi^\lor,\eta^\lor)$. Note that the lines 
    $\ell_{\mathcal T}$ and $\ell_{\mathcal E}$
    intersect in $(\xi^\lor,\eta^\lor)$
    (see \Cref{fig:ellipse} (left)).
    The result of the lemma will follow if we prove that the slope
    of $\ell_{\mathcal T}$ is greater than the slope of 
    $\ell_{\mathcal E}$
    since then the triangle ${\mathcal T}$ is out of the ellipse ${\mathcal E}$. Both slopes can be easily computed using 
    \eqref{eq:G30_ellipse} and \eqref{eq:line_ell_T} and we have to
    verify the inequality 
    $$
    \sqrt{1-c^2}(12\xi^\lor+5c)-3c\eta^\lor>0.
    $$
    But it follows from the fact that $\sqrt{1-c^2}(12\xi+5c)-3c\eta>0$
    for $(\xi,\eta)\in I\times J$ (we simply check the inequality at the
    corner points of $I\times J$) and the fact that 
    $(\xi^\lor,\eta^\lor)\in I\times J$.
\end{proof}

\begin{figure}[htbp]
    \centering
    \begin{subfigure}{.5\textwidth} 
        \centering
        \includegraphics[scale=0.8]{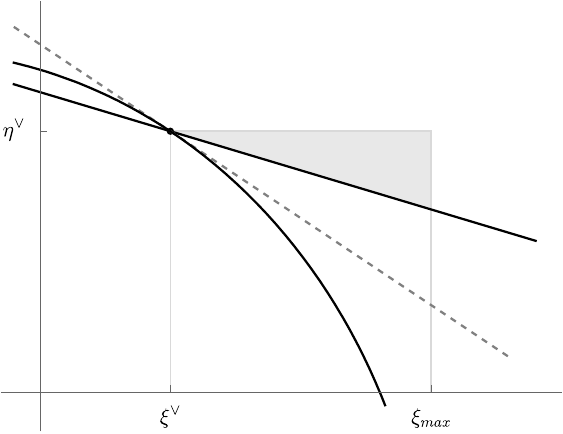}
    \end{subfigure}%
    \begin{subfigure}{.5\textwidth}
         \centering
         \includegraphics[scale=0.8]{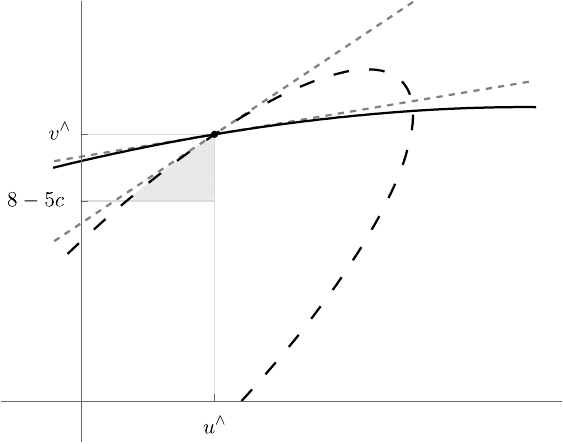}
    \end{subfigure}%
    \caption{Schematic depiction of an ellipse and line segment in proof of lemma \ref{lem:cubicG0} (left) and of parabola and ellipse in proof of lemma \ref{lem:quarticG1} (right).}
    \label{fig:ellipse}
\end{figure}
Let us summarize the results in the following theorem.

\begin{thm}
    The optimal inner cubic $G^0$ approximant of a circular arc is given by
    the pair of parameters $(\xi^\lor,\eta^\lor)\in I\times J$,
    where $\xi^\lor$ is a unique zero of $g_1$ in $I$ and $\eta^\lor$
    is a unique zero of $g_2(\cdot,\xi^\lor)$ in $J$.
\end{thm}

\subsection{Optimal outer approximant}
We will show that the optimal outer approximant is $G^2$ approximant. From the equations $\psi(0,\xi^\land,\eta^\land)=0$ and $\psi'(1,\xi^\land,\eta^\land)=0$ we get $\xi^\land=\tfrac{4-c}3$ and $\eta^\land=\tfrac{3-4 c+c^2}{3 \sqrt{1-c^2}}$. Suppose that a pair $(\overbar{\xi},\overbar{\eta})$ induces a better outer approximant. From $\psi(0,\overbar{\xi},\overbar{\eta})\ge \psi(0,\xi^\land,\eta^\land)=0$ we get $\overbar{\xi}\ge \xi^\land$. From 
$$
1-\sqrt{1-c^2}\overbar{\eta}-c \overbar{\xi}=
\tfrac 1 3\psi'(1,\overbar{\xi},\overbar{\eta})\le 
\tfrac 1 3 \psi'(1,\xi^\land,\eta^\land)=1-\sqrt{1-c^2}\eta^\land-c \xi^\land=0
$$ 
we get $\overbar{\xi}\ge \tfrac{1}{c}\left(\sqrt{1-c^2}(\eta^\land-\overbar{\eta})+c\xi^\land \right)$. From 
\begin{multline*}
\tfrac{1}{108}\left((3\overbar{\eta}+5\sqrt{1-c^2})^2+27(\overbar{\xi}+c)^2-108\right)=\psi(\tfrac{\sqrt{3}}{3},\overbar{\xi},\overbar{\eta})\\
\le\max\psi(\cdot,\xi^\land,\eta^\land)=\psi(\tfrac{\sqrt{3}}{3},\xi^\land,\eta^\land)=\tfrac{1}{108}\left((3\eta^\land+5\sqrt{1-c^2})^2+27(\xi^\land+c)^2-108\right)
\end{multline*}
we get $\overbar{\eta}\le \eta^\land$. From the last two inequalities we get
\begin{align*}
 0&\le (3\eta^\land+5\sqrt{1-c^2})^2+27(\xi^\land+c)^2-(3\overbar{\eta}+5\sqrt{1-c^2})^2-27(\overbar{\xi}+c)^2\\
 &\ge (3\eta^\land+5\sqrt{1-c^2})^2+27(\xi^\land+c)^2-(3\overbar{\eta}+5\sqrt{1-c^2})^2-27(\tfrac{1}{c}\left(\sqrt{1-c^2}(\eta^\land-\overbar{\eta})+c\xi^\land \right)+c)^2\\
 &=-\tfrac{3}{c^2}\left(\eta^\land-\overbar{\eta}\right)\left((9-6c^2)(\eta^\land-\overbar{\eta})+2c^2(4\sqrt{1-c^2}-3\eta^\land)+18c\sqrt{1-c^2}\xi^\land \right)\ge0.
\end{align*}
The last inequality follows since $\overbar{\eta}\le \eta^\land\le \tfrac{4}{3}\sqrt{1-c^2}$. Hence all inequalites are equalities, so $(\overbar{\xi},\overbar{\eta})=(\xi^\land,\eta^\land)$.

\section{Optimal quartic $G^1$ approximants}
\label{sec:G41}
In the $G^1$ quartic case the control points of an approximant $\bfm{p}$ are given as 
\begin{equation*}
    \bfm{b}_0 = (c,-s)^T, \bfm{b}_1 = (c,-s)^T+\xi (s,c)^T, 
    \bfm{b}_2=\left(\eta,0\right)^T,
    \bfm{b}_3 = (c,s)^T+\xi(s,-c)^T, \bfm{b}_4 = (c,s)^T,
\end{equation*}
where $\xi > 0$ due to the $G^1$ continuity.
The function $\psi$ is again a polynomial of degree $8$ and can be written as
\begin{equation*}
    \label{quartic1_psi}
    \psi(t,\xi,\eta)=(t^2-1)^2\left(a_2(\xi,\eta)t^4+a_1(\xi,\eta)t^2+a_0(\xi,\eta)\right),
\end{equation*}
where
\begin{align*}
    a_2(\xi,\eta)&=\text{lc}(\psi(\cdot,\xi,\eta))= \frac{1}{64}\left(3\eta-4\sqrt{1-c^2}\xi-3c\right)^2,\\
    a_1(\xi,\eta)&= \frac{1}{64}\left(32c^2\xi^2+32\xi^2-64c\sqrt{1-c^2}\xi -18\eta^2+36c\eta-34c^2+16\right),\\
    a_0(\xi,\eta)&= \psi(0,\xi,\eta) = \frac{1}{64}\left((3\eta+4\sqrt{1-c^2}\xi+5c)^2 -64\right),
\end{align*}
and its second derivative is $\psi''(1,\xi,\eta) = 8\xi^2+6c\eta-6$.

\subsection{Optimal inner approximant}
It will turn out that the optimal inner quartic $G^1$ approximant is
actually the optimal $G^2$ approximant from 
\Cref{subsec:G24_inner}. Recall that, in this case, we have
$\psi(0,\xi,\eta)=0$ and $\psi''(1,\xi,\eta)=0$,
and the appropriate solution which induces the best approximant
is given by
\begin{equation*}
  \xi^\lor=\frac{c \left(1-c^2\right)+(1-c)^2 \sqrt{3+4 c+c^2}}
  {2 \sqrt{1-c^2}},\quad 
  \eta^\lor=\frac{1}{3} \left(8-7 c+2 c^3-2 (1-c)^2 \sqrt{3+4 c+c^2}\right),
\end{equation*}
with the error $\frac{1}{1024}\left(27 (1-c)^4 \left((2+c)-\sqrt{3+4 c+c^2}\right)^2\right)\le \frac{27 \left(2-\sqrt{3}\right)^2}{1024}$.
Suppose that a pair $(\overbar{\xi},\overbar{\eta})\ne (\xi^\lor,\eta^\lor)$ induces the optimal inner $G^1$ approximant. By \Cref{subsec:G24_inner}, we know that the induced approximant is not a $G^2$ approximant. First, we show $\overbar{\eta}\ge 0$.

Suppose that $\overbar{\eta}<0$. Since $\psi(0,\overbar{\xi},\overbar{\eta})\ge \min_{t\in[-1,1]} \psi(\cdot,\xi^\lor,\eta^\lor)=:m$, we have $3\overbar{\eta}+4\sqrt{1-c^2}\overbar{\xi}+5c<-8\sqrt{1+m}<-2$. Then
\begin{multline*}
\text{lc}(\psi(\cdot,\overbar{\xi},\overbar{\eta}))=\frac{1}{64}\left(-3\overbar{\eta}+4\sqrt{1-c^2}\overbar{\xi}+3c\right)^2\ge \frac{1}{64}\left(8\sqrt{1-c^2} \overbar{\xi}+8c+2\right)^2
\ge \frac{1}{64}\left(8c+2\right)^2\\
>\frac{1}{4} (1-c)^4 \left(7+8 c+2 c^2-2 (2+c) \sqrt{3+4 c+c^2}\right)=\text{lc}(\psi(\cdot,\xi^\lor,\eta^\lor)).
\end{multline*}
Together with the inequality $\psi''(1,\overbar{\xi},\overbar{\eta})<\psi''(1,\xi^\lor,\eta^\lor)=0$ we get that $\psi(t,\overbar{\xi},\overbar{\eta})<\psi(t,\xi^\lor,\eta^\lor)<0$
for all $t$ close but not equal to 1. Hence the graphs of $\psi(\cdot,\overbar{\xi},\overbar{\eta})$ and $\psi(\cdot,\xi^\lor,\eta^\lor)$ have one intersection on $(1,\infty)$ and one on $(-\infty,-1)$. Since they have at least double zeros at $\pm 1$, they have at most two intersections on $(-1,1)$, so there are no intersections on $(-1,1)$. Hence $(\overbar{\xi},\overbar{\eta})$ induces a worse $G^1$ inner approximant then $(\xi^\lor,\eta^\lor)$, hence $\overbar{\eta}\ge 0$.

Since $\overbar{\eta}\ge 0$ and $\psi(t,\overbar{\xi},\overbar{\eta})=b_2 {\overbar{\eta}}^2+b_1 \overbar{\eta}+b_0$, where 
$$
b_2=\tfrac 9{64}(1-t^2)^4> 0,\quad b_1=\tfrac{3}{32} \left(1-t^2\right)^2 \left(c \left(5+6 t^2-3 t^4\right)+4 \sqrt{1-c^2}\left(1-t^4\right) \overbar{\xi} \right)>0,
$$
for all $t\in (-1,1)$, the function $\psi(t,\cdot,\overbar{\eta})$ is increasing. If $\psi(0,\overbar{\xi},\overbar{\eta})<0$ we can enlarge $\overbar{\xi}$ and get better inner $G^1$ approximant. Hence $\psi(0,\overbar{\xi},\overbar{\eta})=0$ and therefore $\overbar{\eta}=\eta(\overbar{\xi})=\tfrac 1 3(8-4\sqrt{1-c^2}\overbar{\xi}-5c)$. If $\overbar{\xi}>\xi^\lor$ then $\psi(0,\overbar{\xi},\eta(\overbar{\xi}))>0$, so $(\overbar{\xi},\eta(\overbar{\xi}))$ does not induce an inner approximant. If $\overbar{\xi}<\xi^\lor$, then $\psi(\tfrac 1 2,\overbar{\xi},\eta(\overbar{\xi}))<\psi(\tfrac 1 2,\xi^\lor,\eta^\lor)=\min_{t\in[-1,1]} \psi(\cdot,\xi^\lor,\eta^\lor)$, hence $(\overbar{\xi},\overbar{\eta})$ induces worse approximant then $(\xi^\lor,\eta^\lor)$.

\subsection{Optimal outer approximant}
To simplify the coefficients of the polynomial $\psi$, 
we introduce two new variables
\begin{equation*}
    u = 3\eta-4\sqrt{1-c^2}\xi,\quad v = 3\eta+4\sqrt{1-c^2}\xi
\end{equation*}
which implies
$\xi = \tfrac{v-u}{8\sqrt{1-c^2}}$ and $\eta = \tfrac{1}{6}(u+v)$.
Now $\psi$ can be written as $\psi_{uv}(t,u,v)=\psi(t,\xi(u,v),\eta(u,v))$ and 
we get
\begin{equation}
    \label{eq:psi_lc_0_dd_quartic1_uv}
    \text{lc}(\psi_{uv}(\cdot,u,v))= \frac{1}{64}\left(u-3c\right)^2,\quad
    \psi_{uv}(0,u,v) = \frac{1}{64}\left((v+5c)^2 -64\right), \quad
    \psi_{uv}''(1,u,v) = \frac{(v-u)^2}{8(1-c^2)}+c(v+u)-6.
\end{equation}
Observe that $\xi>0$ due to the $G^1$ interpolation condition
and $\psi_{uv}(0,u,v)\geq 0$ since the approximant is the outer one. This implies
\begin{equation}\label{eq:uv_domain}
  v> u \quad{\rm and}\quad \left|v+5c\right|\geq 8.
\end{equation}
Let us first assume that $v\geq 8-5c$.\\
It is clear that $\psi_{uv}$ is symmetric and has two double zeros at $\pm 1$.
Following the idea of equioscillation of the best polynomial approximant in the functional
case, we might expect that $\psi_{uv}$ equioscillates. Since it is an outer
approximant, it may touch the abscissa several times. Thus, the guess is that its
graph looks like on \Cref{fig:graph_psi_cubic_quartic} (right). But then its four double
zeros on $[-1,1]$ must be related to the minima of the Chebyshev polynomial of degree eight.
Consequently, the two double zeros on $(-1,1)$ are equal to $1-\sqrt{2}$ and
$\sqrt{2}-1$. Since these points are also minima of $\psi_{uv}$, we must have
\begin{equation}\label{eq:system_quartic1_uv}
        \psi_{uv}\left(\sqrt{2}-1,u,v\right) = 0, \quad  
        \psi_{uv}'\left(\sqrt{2}-1,u,v\right) = 0. \quad
\end{equation}
This leads to the system of two algebraic equations for $u$ and $v$.
Its Gr\"obner basis according to the ordering $v\prec u$
are polynomials $p_1$ and $p_2$, given as
\begin{equation}
    \label{eq:polinomG}
    p_1(u) = \frac{1}{1-c^2} \left(\alpha_4u^4+\alpha_3u^3+\alpha_2u^2+\alpha_1u+\alpha_0\right),
\end{equation}
where 
\begin{align}
    \alpha_4&=c^2(c^2+1),\nonumber\\  
    \alpha_3&=-8(3+\sqrt{2})c^5+4(4+5\sqrt{2})c^3+4(4+3\sqrt{2})c,\nonumber\\  
    \alpha_2&=2(125+68\sqrt{2})c^6-2(229+182\sqrt{2})c^4+4(140+97\sqrt{2})c^2-4(65+46\sqrt{2}),\label{eq:alphas_p1}\\
    \alpha_1&= -88(14+9\sqrt{2})c^7+12(116+89\sqrt{2})c^5-84(16+11\sqrt{2})c^3
    +4(50+36\sqrt{2})c,\nonumber\\
    \alpha_0&= 5(457+312\sqrt{2})c^8-(1655+1236\sqrt{2})c^6
    +12(16+9\sqrt{2})c^4+4(163+114\sqrt{2})c^2+272+192\sqrt{2},\nonumber
\end{align}
and
\begin{equation*}
    p_2(u,v) = \frac{1}{c^2-1}\left(\beta_1 v+\beta_0(u)\right),
\end{equation*}
where
\begin{align*}
    \beta_1&= -2 \left(
    4 c^8
    +4\left(4+\sqrt{2}\right) c^6
    -\left(1+6 \sqrt{2}\right) c^4
    +2 c^2
    +3+2\sqrt{2}
    \right),\\
    \beta_0(u)&= 36 \left(3-2 \sqrt{2}\right) c^9+16 \left(3 \sqrt{2}-5\right) c^8
   u+\left(329-198 \sqrt{2}+\left(12-8 \sqrt{2}\right) u^2\right) c^7+\left(142
   \sqrt{2}-253\right) c^6 u\\
   &+\left(169-198 \sqrt{2}+\left(55-34 \sqrt{2}\right)
   u^2\right) c^5+\left(174 \sqrt{2}-51+\left(2 \sqrt{2}-3\right) u^2\right) c^4
   u+\left(4 \left(43+9 \sqrt{2}\right)-5 \left(6 \sqrt{2}-7\right) u^2\right)
   c^3\\
   &-\left(16 \left(7+3 \sqrt{2}\right) u+\left(3-2 \sqrt{2}\right) u^3\right)
   c^2+2 \left(23+16 \sqrt{2}\right) u-4 c \left(10+9 \sqrt{2}+\sqrt{2}
   u^2\right).
\end{align*}

\begin{figure}
    \centering
    \includegraphics[scale=0.9]{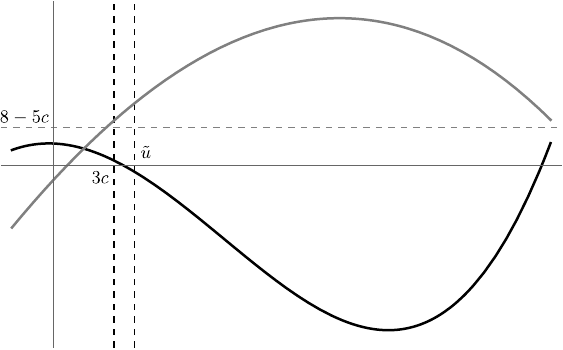}
    \caption{Graphs of polynomials $p_1$ (black) and $-\tfrac{\beta_0}{\beta_1}$ (gray).}
    \label{fig:graphp1p2}
\end{figure}
\noindent Let us define the domain of interest for $(u,v)$ as
\begin{equation*}
    {\mathcal D}=(3c,\Tilde{u})\times[8-5c,\infty),
\end{equation*}
where
  \begin{equation*}
      \Tilde{u}=kc^2-(2k-3)c+k, \quad k=2 \sqrt{\frac{2\sqrt{2}-1}{7}}.
  \end{equation*}
Note that ${\mathcal D}$ is a subset of \eqref{eq:uv_domain}.
We shall prove in the following that there exists a unique 
solution $(u^\land,v^\land)\in{\mathcal D}$ of \eqref{eq:system_quartic1_uv} 
with $u^\land$ the unique zero of $p_1$ on $I=(3c,\Tilde{u})$
and 
$v^\land=-\tfrac{\beta_0(u^\land)}{\beta_1}\geq 8-5c$.
Moreover, this solution is {\it optimal}, i.e., it induces
the polynomial approximant $\bfm{p}$ with 
$\psi_{u^\land v^\land}$ for which 
\begin{equation*}
  (u^\land,v^\land)=
  \underset{\substack{(u,v)}}{\rm argmin}
  \left(\max_{t\in[-1,1]}\left|\psi_{uv}(t,u,v)\right|\right).
\end{equation*}
Let us first confirm the following lemma.

\begin{lem}\label{lem:uhat}
    The polynomial $p_1$ has precisely one zero $u^\land$  
    on $I$, and the system \eqref{eq:system_quartic1_uv}
    has precisely one solution $(u^\land,v^\land)$ in ${\mathcal D}$.
\end{lem}

\begin{proof}
    By \eqref{eq:polinomG} and \eqref{eq:alphas_p1}
    the leading coefficient of $p_1$ is 
    $\tfrac{c^4+c^2}{1-c^2}$, thus
    $\lim_{u\to \pm \infty}p_1(u)=\infty$.
    A computer algebra system reveals that
    $(1+c)p_1(-6\sqrt{2}-5)=\sum_{i=0}^7 a_i c^i(1-c)^{7-i}$,
    $(1-c^2)p_1(0)=\sum_{i=0}^4 b_i c^{2i}(1-c^2)^{4-i}$, 
    $p_1(3c)=16 \left(17+12 \sqrt{2}\right) \left(1-c^2\right)^3$, and
    $49(1+c)p_1(\Tilde{u})=16c(1-c)^3\sum_{i=0}^{7} d_i c^i(1-c)^{7-i}$,
    where $a_i<0$, $i=0,1,\dots,7$, $b_i>0$, $i=0,1,\dots,4$, and 
    $d_i<0$, $i=0,1,\dots,7$. Thus
    $p_1(-6\sqrt{2}-5)<0$, $p_1(0)>0$, $p_1(3c)>0$, and
    $p_1(\Tilde{u})<0$ and the result of the lemma follows
    since $p_1$ is quartic.
\end{proof}
\noindent Let us confirm that the solution $(u^\land,v^\land)$ induces the optimal
approximant. If not, then there exists a pair
  $(\overbar{u},\overbar{v})\in {\mathcal D}$, $(\overbar{u},\overbar{v})\neq(u^\land,v^\land)$, 
  for which
  $$
    \max_{t\in[-1,1]}\left|\psi_{uv}(t,\overbar{u},\overbar{v})\right|
   <
    \max_{t\in[-1,1]}\left|\psi_{uv}(t,u^\land,v^\land)\right|.
  $$
  If $\overbar{v}>v^\land$, then 
  $\psi_{uv}(0,\cdot,\overbar{v})>\psi_{uv}(0,\cdot,v^\land)$ since $\overbar{v}>8-5c$
  by \eqref{eq:uv_domain}, and $(u,\overbar{v})$ can not induce better approximant than
  $(u^\land,v^\land)$. Thus $\overbar{v}\leq v^\land$. 
  Consider first the case $\overbar{u}>u^\land$. By \Cref{lem:uhat}, $u^\land>3c$ and 
  \eqref{eq:psi_lc_0_dd_quartic1_uv} implies
  $\text{lc}(\psi_{uv}(t,\overbar{u},\overbar{v})) 
   > 
   \text{lc}(\psi_{uv}(t,u^{\land},v^{\land}))$.
   If $\psi_{uv}''(1,u^{\land},v^{\land})
   >
   \psi_{uv}''(1,\overbar{u},\overbar{v})$, then the graphs of $\psi_{uv}(\cdot,u^{\land},v^{\land})$ and 
   $\psi_{uv}(\cdot,\overbar{u},\overbar{v})$ intersect on $(1,\infty)$. Moreover, if $(\overbar{u},\overbar{v})$ induces a better approximant as
   $(u^\land,v^\land)$, then the graphs of $\psi_{uv}(\cdot,u^{\land},v^{\land})$ and 
   $\psi_{uv}(\cdot,\overbar{u},\overbar{v})$ would have to intersect at least eight times on $[-1,1]$,
   which is not possible since 
   $\psi_{uv}(\cdot,\overbar{u},\overbar{v})$ 
   is a polynomial of degree
   at most $8$ (see \Cref{fig:graph_psi_intersections} (right)).
   Similarly, if $\psi_{uv}''(1,u^{\land},v^{\land})
   \leq
   \psi_{uv}''(1,\overbar{u},\overbar{v})$, the above-mentioned graphs would have to 
    intersect at least ten times on $[-1,1]$ which is again a contradiction.
   Finally, let us consider the case $\overbar{v}\leq v^{\land}$ and 
   $\overbar{u} \leq u^{\land}$. 
If $\psi_{uv}''(1,u^{\land},v^{\land})\leq
\psi_{uv}''(1,\overbar{u},\overbar{v})$, then the graphs of $\psi_{uv}(\cdot,u^{\land},v^{\land})$ and 
$\psi_{uv}(\cdot,\overbar{u},\overbar{v})$ would have to intersect at least ten times on $[-1,1]$ which is again an obvious contradiction. 
So, we are left with the case 
$\psi_{uv}''(1,u^{\land},v^{\land})>\psi_{uv}''(1,\overbar{u},\overbar{v})$. The following lemma reveals
that, in this case, the approximant is not an outer one.

\begin{lem}
\label{lem:quarticG1}
    Let $\overbar{v}\leq v^{\land}$ and $\overbar{u}\leq u^{\land}$. If $\psi_{uv}''(1,u^{\land},v^{\land})
    >\psi_{uv}''(1,\overbar{u},\overbar{v})$, then 
    $\psi_{uv}(\sqrt{2}-1,\overbar{u},\overbar{v})<0$.
\end{lem}

\begin{proof}
   Let us define
   \begin{equation}\label{eq:fg}
       f(u,v)=\psi''_{uv}(1,u,v)-\psi''_{uv}(1,u^\land,v^\land),
       \quad
       g(u,v)=\psi_{uv}(\sqrt{2}-1,u,v).
   \end{equation}
   Observe that $f(u,v)=0$ is a parabola and $g(u,v)=0$ is an ellipse (see \Cref{fig:ellipse}, right)
   with leading coefficients of their tangent lines at $(u,v)$ equal to
   \begin{align*}
        c_p(u,v)&=\frac{-4c^3+4c+u-v}{4c^3-4c+u-v},\\
        c_e(u,v)&=\frac{
        (12\sqrt{2}-17)u
        +2c(15-11\sqrt{2}+(11\sqrt{2}-15)c^2
        +(7-5\sqrt{2})cu)+(3-2\sqrt{2})v}
        {(2 \sqrt{2}-3)u+v+2c(1+\sqrt{2}-c((1+\sqrt{2})c+(\sqrt{2}
        -1)v))},
    \end{align*}
  respectively. With the help of a computer algebra system, it can be again shown that
  $c_p(u,v),c_e(u,v)>0$ and $c_p(u,v)>c_e(u,v)$
  for $(u,v)\in{\mathcal D}$.
  Thus, by \eqref{eq:uv_domain} and \Cref{lem:uhat}, all these inequalities 
  hold for $u=u^\land$ and $v=v^\land$. 
  The parabola and the ellipse intersect at $(u^\land,v^\land)$, and
  $f(0,0)= - \tfrac{1}{8(1-c^2)}(v^{\land} - u^{\land})^2 
  - c(v^{\land} + u^{\land})<0$, 
  $g(0,0)=-\tfrac{16\sqrt{2}-20+(68\sqrt{2}-97)c^2}{4}<0$.
  Since, in addition, the ordinate of the intersection of the tangent line to ellipse with $v$-axis is positive (which can be again
  checked by computer algebra system), and the parabola is
  by \eqref{eq:psi_lc_0_dd_quartic1_uv} and \eqref{eq:fg}
  symmetric about the line $v=u$, the parabola and ellipse are concave at $(u^\land,v^\land)$. This implies
  $\psi_{uv}(\sqrt{2}-1,\overbar{u},\overbar{v})<0$ and the result of the lemma is confirmed.
\end{proof}

Finally, consider $v\leq -8-5c$. By \eqref{eq:uv_domain} we have $u<v$, therefore 
$u<-8-5c$, and $\text{lc}(\psi_{uv}(\cdot,u,v))=\tfrac{1}{64}(u-3c)^2\ge \tfrac{1}{64}(-8-5c-3c)^2=(1+c)^2$. 
From the above calculation we have $\text{lc}(\psi_{uv}(\cdot,u^\land,v^\land))\le \tfrac{1}{64}(\Tilde{u}-3c)^2=\frac{1}{112} \left(2 \sqrt{2}-1\right) (1-c)^4\le \text{lc}(\psi_{uv}(\cdot,u,v))$. Hence, if the pair $(u,v)$ induces a better outer approximant than the pair $(u^\land,v^\land)$, then the error graphs of induced approximants have at least ten intersections, which is impossible.
We can now finally formulate the main result of this subsection.
\begin{thm}
    The optimal quartic $G^1$ outer approximant is given by the unique solution of
    the nonlinear algebraic system \eqref{eq:system_quartic1_uv} on 
    the domain ${\mathcal D}$.
\end{thm}
\begin{rem}
   Note that although the zeros of $p_1$ given by \eqref{eq:polinomG} can
   be found analytically (since $p_1$ is a quartic polynomial), 
   it is much more efficient to compute the appropriate zero on $(3c,\tilde{u})$
   by some numerical method since the interval of interest provides a good 
   starting point.
\end{rem}


\begin{figure}[htbp]
    \centering
    \begin{subfigure}{.5\textwidth}
         \centering
         \includegraphics[scale=0.7]{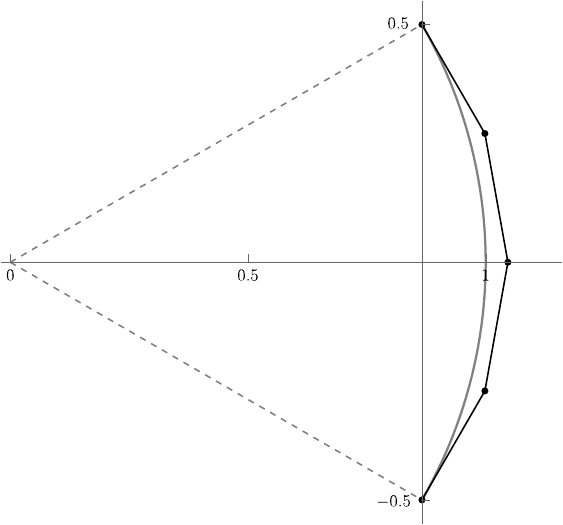}
    \end{subfigure}%
    \begin{subfigure}{.5\textwidth}
         \centering
         \includegraphics[scale=0.7]{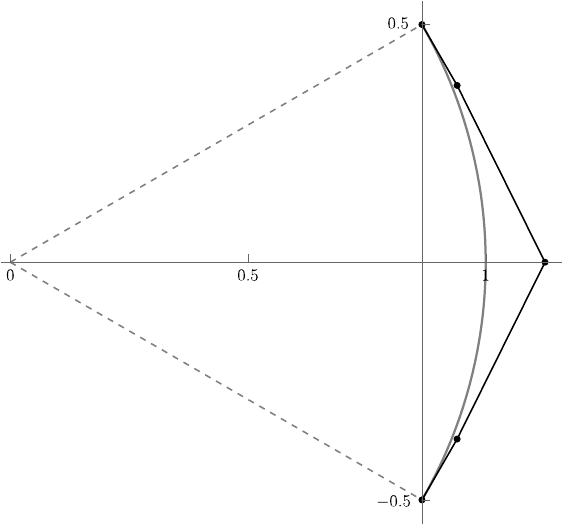}
    \end{subfigure}%
    \caption{The optimal (left) and the second best (right)
    outer quartic $G^1$ approximants of the circular arc given
    by the inner angle 
    $2\varphi = \frac{\pi}{3}$.
    The radial distances are $3.80\times 10^{-8}$
    and $5.31\times 10^{-5}$, respectively.}
    \label{fig:G41_one_two}
\end{figure}

\begin{figure}[htbp]
    \centering
    \begin{subfigure}{.4\textwidth} 
        \centering
        \includegraphics[scale=0.8]{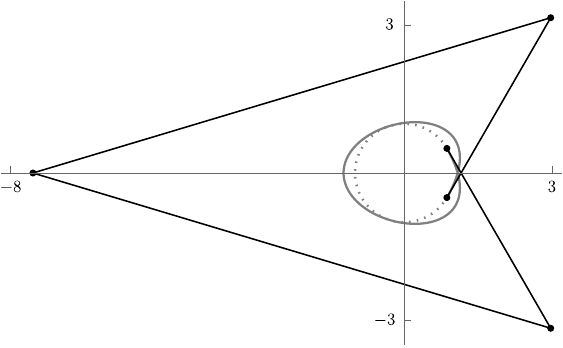}
    \end{subfigure}%
    \begin{subfigure}{.6\textwidth}
         \centering
         \includegraphics[scale=0.8]{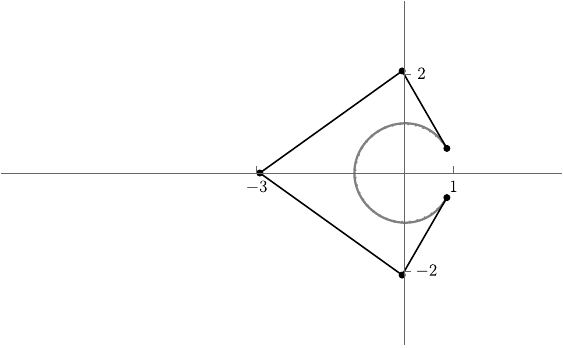}
    \end{subfigure}%
    \caption{Two additional outer quartic $G^1$
    approximants of the circular arc given by an inner angle
    $2\varphi = \frac{\pi}{3}$. They approximate
    the complementary circular arc with radial distances
    $2.34\times 10^{-1}$ and $1.38\times 10^{-2}$, respectively.}
    \label{fig:G41_three_four}
\end{figure}

\section{Numerical examples}\label{sec:examples}

In this section, several numerical examples will be presented
to confirm the theoretical results derived in the previous 
sections. The quality of optimal inner and outer 
approximants will be compared to the general 
optimal approximants obtained in 
\cite{Vavpetic-Zagar-2019-framework-circle-arcs} where 
the Hausdorff distance has been considered as a measure of
quality. Recall that in the case of inner and outer approximants,
the radial distance $|\psi_r|$ coincides with the 
Hausdorff one, too.\\
For the parabolic $G^0$ case, the optimal
parameters and radial 
distances for several inner angles of a circular arc 
are collected in \Cref{tab:G20}. The optimal parameters and
corresponding radial distances for the cubic $G^0$ and 
the cubic $G^1$ case are
in \Cref{tab:G30} and \Cref{tab:G31}, respectively.\\
As we have seen in \Cref{sec:G41}, the quartic $G^1$ case is
probably the most interesting. Optimal parameters with
the corresponding radial distances are in \Cref{tab:G41}.
There always exist four outer admissible solutions arising
from the zeros of $p_1$ given by \eqref{eq:polinomG}
but they provide quite different approximants. The best two are on
\Cref{fig:G41_one_two}. They look quite similar, but the
radial distances differ significantly. The other two approximants
correspond to two additional sets of parameters. Their graphs
are presented on \Cref{fig:G41_three_four}. It is seen that they
represent an approximation of the complementary circular arc.
This is because the radial error 
at a particular point on the curve is measured in the
radial direction from this point to the part of the circle which 
might not necessarily include the considered circular arc).
We guess that these solutions imply the optimal approximants
of the complementary arc (see \cite{Vavpetic-2020-optimal-circular-arcs}).\\
Finally, the optimal parameters with the corresponding
radial distances for the quartic $G^2$ case are collected in
\Cref{tab:G42}.

\begin{table}[htbp]
\begin{equation*}
\begin{array}{|c|c|c|c|c|c|c|}
\hline
\varphi  & \xi ^* & \text{error} & \xi ^{\land } & \text{error} & \xi ^{\lor }
   & \text{error} \\ \hline
 \frac{\pi }{2} & 2.21535 & 1.08\times 10^{-1} & \text{no approximant} &
   \text{--} & 2. & 1.34\times 10^{-1} \\ \hline
 \frac{\pi }{3} & 1.54728 & 2.36\times 10^{-2} & 2. & 2.5\times 10^{-1} & 1.5 &
   3.18\times 10^{-2} \\ \hline
 \frac{\pi }{4} & 1.30843 & 7.77\times 10^{-3} & 1.41421 & 6.07\times 10^{-2} &
   1.29289 & 1.08\times 10^{-2} \\ \hline
 \frac{\pi }{6} & 1.13713 & 1.58\times 10^{-3} & 1.1547 & 1.04\times 10^{-2} &
   1.13397 & 2.25\times 10^{-3} \\ \hline
 \frac{\pi }{8} & 1.07713 & 5.04\times 10^{-4} & 1.08239 & 3.14\times 10^{-3} &
   1.07612 & 7.25\times 10^{-4} \\ \hline
 \frac{\pi }{12} & 1.03427 & 1.\times 10^{-4} & 1.03528 & 6.01\times 10^{-4} &
   1.03407 & 1.45\times 10^{-4} \\ \hline
\end{array}
\end{equation*}
\caption{The optimal parameters $\xi^*$, $\xi^\land$, and $\xi^\lor$ for parabolic $G^0$ approximants according to the radial
distance
for several different inner angles $2\varphi$ of a circular arc.}
  \label{tab:G20}
\end{table}

\begin{table}[htbp]
\begin{equation*}
\begin{array}{|@{\;}c@{\;}|@{\;}c@{\;}|@{\;}c@{\;}|@{\;}c@{\;}|@{\;}c@{\;}|@{\;}c@{\;}|@{\;}c@{\;}|@{\;}c@{\;}|@{\;}c@{\;}|@{\;}c@{\;}|}
\hline
 \varphi  & \xi ^* & \eta ^* & \text{error} & \xi ^{\land } & \eta ^{\land } &
   \text{error} & \xi ^{\lor } & \eta ^{\lor } & \text{error} \\ \hline
 \frac{\pi }{2} & 1.32801 & 0.940495 & 3.99\times 10^{-3} & 1.33333 & 1. &
   1.84\times 10^{-2} & 1.32508 & 0.925926 & 6.19\times 10^{-3} \\ \hline
 \frac{\pi }{3} & 1.16617 & 0.474943 & 3.75\times 10^{-4} & 1.16667 & 0.481125
   & 1.54\times 10^{-3} & 1.16587 & 0.473285 & 5.99\times 10^{-4} \\ \hline
 \frac{\pi }{4} & 1.09754 & 0.315229 & 6.84\times 10^{-5} & 1.09763 & 0.316582
   & 2.73\times 10^{-4} & 1.09748 & 0.31486 & 1.1\times 10^{-4} \\ \hline
 \frac{\pi }{6} & 1.04465 & 0.190431 & 6.11\times 10^{-6} & 1.04466 & 0.190599
   & 2.39\times 10^{-5} & 1.04465 & 0.190384 & 9.89\times 10^{-6} \\ \hline
 \frac{\pi }{8} & 1.02537 & 0.137616 & 1.09\times 10^{-6} & 1.02537 & 0.137655
   & 4.25\times 10^{-6} & 1.02537 & 0.137605 & 1.77\times 10^{-6} \\ \hline
 \frac{\pi }{12} & 1.01136 & 0.0892586 & 9.65\times 10^{-8} & 1.01136 &
   0.0892636 & 3.73\times 10^{-7} & 1.01136 & 0.0892572 & 1.57\times 10^{-7} \\ \hline
\end{array}
\end{equation*}
\caption{The table of the optimal parameters $(\xi^*,\eta^*)$, $(\xi^\land,\eta^\land)$, and $(\xi^\lor,\eta^\lor)$ for cubic $G^0$ approximants according to the radial distance for several different inner angles $2\varphi$ of a circular arc.}
  \label{tab:G30}
\end{table}

\begin{table}[htbp]
\begin{equation*}
\begin{array}{|c|c|c|c|c|c|c|}
\hline
 \varphi  & \xi ^* & \text{error} & \xi ^{\land } & \text{error} & \xi ^{\lor }
   & \text{error} \\ \hline
 \frac{\pi }{2} & 1.31574 & 1.32\times 10^{-2} & 1.33333 & 1.84\times 10^{-2} &
   1.1547 & 1.34\times 10^{-1} \\ \hline
 \frac{\pi }{3} & 0.768087 & 1.11\times 10^{-3} & 0.7698 & 1.54\times 10^{-3} &
   0.752158 & 1.15\times 10^{-2} \\ \hline
 \frac{\pi }{4} & 0.551915 & 1.96\times 10^{-4} & 0.552285 & 2.73\times 10^{-4}
   & 0.548584 & 1.96\times 10^{-3} \\ \hline
 \frac{\pi }{6} & 0.35722 & 1.71\times 10^{-5} & 0.357266 & 2.39\times 10^{-5}
   & 0.356822 & 1.66\times 10^{-4} \\ \hline
 \frac{\pi }{8} & 0.265206 & 3.04\times 10^{-6} & 0.265216 & 4.25\times 10^{-6}
   & 0.265115 & 2.92\times 10^{-5} \\ \hline
 \frac{\pi }{12} & 0.175535 & 2.67\times 10^{-7} & 0.175537 & 3.73\times
   10^{-7} & 0.175524 & 2.54\times 10^{-6} \\ \hline
\end{array}
\end{equation*}
\caption{The table of the optimal parameters $\xi^*$, $\xi^\land$, and $\xi^\lor$ for cubic $G^1$ approximants according to the radial distance for several different inner angles $2\varphi$ of a circular arc.}
  \label{tab:G31}
\end{table}

\begin{table}[htbp]
\begin{equation*}
\begin{array}{|@{\;}c@{\;}|@{\;}c@{\;}|@{\;}c@{\;}|@{\;}c@{\;}|@{\;}c@{\;}|@{\;}c@{\;}|@{\;}c@{\;}|@{\;}c@{\;}|@{\;}c@{\;}|@{\;}c@{\;}|}
\hline
 \varphi  & \xi ^* & \eta ^* & \text{error} & \xi ^{\land } & \eta ^{\land } &
   \text{error} & \xi ^{\lor } & \eta ^{\lor } & \text{error} \\ \hline
 \frac{\pi }{2} & 0.871525 & 1.50505 & 1.57\times 10^{-4} & 0.87247 & 1.50401 &
   2.4\times 10^{-4} & 0.866025 & 1.51197 & 9.47\times 10^{-4} \\ \hline
 \frac{\pi }{3} & 0.547788 & 1.20082 & 6.21\times 10^{-6} & 0.547886 & 1.20071
   & 9.59\times 10^{-6} & 0.547225 & 1.20145 & 3.59\times 10^{-5} \\ \hline
 \frac{\pi }{4} & 0.402721 & 1.10847 & 6.25\times 10^{-7} & 0.402742 & 1.10845
   & 9.69\times 10^{-7} & 0.402599 & 1.10858 & 3.56\times 10^{-6} \\ \hline
 \frac{\pi }{6} & 0.264731 & 1.0468 & 2.45\times 10^{-8} & 0.264734 & 1.0468 &
   3.8\times 10^{-8} & 0.264716 & 1.04681 & 1.38\times 10^{-7} \\ \hline
 \frac{\pi }{8} & 0.197581 & 1.02605 & 2.46\times 10^{-9} & 0.197582 & 1.02605
   & 3.82\times 10^{-9} & 0.197577 & 1.02605 & 1.38\times 10^{-8} \\ \hline
 \frac{\pi }{12} & 0.131263 & 1.01149 & 9.6\times 10^{-11} & 0.131264 & 1.01149
   & 1.49\times 10^{-10} & 0.131263 & 1.01149 & 5.36\times 10^{-10} \\ \hline
\end{array}
\end{equation*}
\caption{The table of the optimal parameters $(\xi^*,\eta^*)$, $(\xi^\land,\eta^\land)$, and $(\xi^\lor,\eta^\lor)$ for quartic $G^1$ approximants according to the radial distance for several different inner angles $2\varphi$ of a circular arc.}
  \label{tab:G41}
\end{table}

\begin{table}[htbp]
\begin{equation*}
\begin{array}{|c|c|c|c|c|c|c|}
\hline
 \varphi  & \xi ^* & \text{error} & \xi ^{\land } & \text{error} & \xi ^{\lor }
   & \text{error} \\ \hline
 \frac{\pi }{2} & - & 6.95\times 10^{-4} & 0.866025 & 1.04\times 10^{-2} &
   0.866025 & 9.47\times 10^{-4} \\ \hline
 \frac{\pi }{3} & 0.547186 & 2.62\times 10^{-5} & 0.546677 & 3.62\times 10^{-4}
   & 0.547225 & 3.59\times 10^{-5} \\ \hline
 \frac{\pi }{4} & 0.402587 & 2.59\times 10^{-6} & 0.402437 & 3.5\times 10^{-5}
   & 0.402599 & 3.56\times 10^{-6} \\ \hline
 \frac{\pi }{6} & 0.264714 & 1.\times 10^{-7} & 0.264692 & 1.33\times 10^{-6} &
   0.264716 & 1.38\times 10^{-7} \\ \hline
 \frac{\pi }{8} & 0.197577 & 1.\times 10^{-8} & 0.197572 & 1.32\times 10^{-7} &
   0.197577 & 1.38\times 10^{-8} \\ \hline
 \frac{\pi }{12} & 0.131263 & 3.9\times 10^{-10} & 0.131262 & 5.1\times 10^{-9}
   & 0.131263 & 5.36\times 10^{-10} \\ \hline
\end{array}
\end{equation*}
\caption{The table of the optimal parameters $\xi^*$, $\xi^\land$, and $\xi^\lor$ for quartic $G^2$ approximants according to the radial distance for several different inner angles $2\varphi$ of a circular arc.}
  \label{tab:G42}
\end{table}

\section{Conclusion}
\label{sec:conclusion}
In this paper, we presented several optimal one-sided approximants of a circular arc given by an inner angle $\varphi \in (0,\tfrac{\pi}{2}]$. Detailed analyses of parabolic
$G^0$, cubic $G^0$, $G^1$, $G^2$, and
quartic $G^1$, $G^2$ cases have been considered. 
Optimal inner and outer approximants are given
explicitly as a solution of a (at most) quartic algebraic
equation. The uniqueness of the optimal solution was
confirmed in all cases.\\
The problem of constructing the optimal inner and outer
approximants gets increasingly complicated as the gap
between the order of geometric continuity and the degree of
the polynomial is growing. Thus, the topics for future work include the construction of higher-degree one-sided approximants with low-order geometric approximation. The first such case
might be the quartic $G^0$ approximation.\\

{\noindent \sl Acknowledgments.}
During the preparation of the paper, the first author was a junior researcher funded by the Slovenian Research and Innovation Agency program  P1-0294. The Slovenian Research and Innovation Agency program P1-0292 and the grant J1-4031 supported the second author. The third author was partly supported by the program P1-0288 and the grants N1-0137, N1-0237 and J1-3005 by the Slovenian Research and Innovation Agency.


\begin{thebibliography}{1}
\expandafter\ifx\csname url\endcsname\relax
  \def\url#1{\texttt{#1}}\fi
\expandafter\ifx\csname urlprefix\endcsname\relax\def\urlprefix{URL }\fi
\expandafter\ifx\csname href\endcsname\relax
  \def\href#1#2{#2} \def\path#1{#1}\fi

\bibitem{Sir-Feichtinger-Juettler-biarcs-2006}
Z.~Šír, R.~Feichtinger, B.~Jüttler, Approximating curves and their offsets
  using biarcs and pythagorean hodograph quintics, Comput.-Aided Des. 38~(6)
  (2006) 608--618.

\bibitem{Dokken-Daehlen-Lyche-Morken-90-CAGD}
T.~Dokken, M.~D{\ae}hlen, T.~Lyche, K.~M{\o}rken, Good approximation of circles
  by curvature-continuous {B}\'ezier curves, Comput. Aided Geom. Design 7~(1-4)
  (1990) 33--41, {C}urves and surfaces in CAGD '89 (Oberwolfach, 1989).

\bibitem{Vavpetic-2020-optimal-circular-arcs}
A.~Vavpeti\v{c}, Optimal parametric interpolants of circular arcs, Comput.
  Aided Geom. Design 80 (2020) 101891, 9.

\bibitem{Vavpetic-Zagar-21-circle-arcs-Hausdorff}
A.~Vavpeti\v{c}, E.~\v{Z}agar, On optimal polynomial geometric interpolation of
  circular arcs according to the {H}ausdorff distance, J. Comput. Appl. Math.
  392 (2021) Paper No. 113491, 14.

\bibitem{Vavpetic-Zagar-2019-framework-circle-arcs}
A.~Vavpeti\v{c}, E.~\v{Z}agar, A general framework for the optimal
  approximation of circular arcs by parametric polynomial curves, J. Comput.
  Appl. Math. 345 (2019) 146--158.

\bibitem{Jaklic-Kozak-2018-best-circle}
G.~Jakli\v{c}, J.~Kozak, On parametric polynomial circle approximation, Numer.
  Algorithms 77~(2) (2018) 433--450.

\end{thebibliography}

\end{document}